\documentclass[a4paper,10pt]{article}
\usepackage[english]{babel}
\usepackage[utf8]{inputenc} 
\usepackage{amssymb,amsmath,url,xspace,amsthm}
\usepackage[all]{xy}
\usepackage{enumerate}
\usepackage{url}

\theoremstyle{plain}
\newtheorem{thm}{Theorem}[section]

\newtheorem{proposition}[thm]{Proposition}

\newtheorem{corollary}[thm]{Corollary}
\newtheorem{lemma}[thm]{Lemma}
\newtheorem{theorem}[thm]{Theorem}

\theoremstyle{definition}
\newtheorem{definition}[thm]{Definition}
\newtheorem*{remark}{Remark}

\newtheorem{assume}{Assumptions}
\newtheorem{example}[thm]{Example}
\addto\captionsfrench{}

\newcommand{\ov}{\overline}

\newcommand{\surj}{\twoheadrightarrow}

\newcommand{\X}{\mathfrak{X}}
\newcommand{\Z}{\mathbb{Z}}
\newcommand{\Q}{\mathbb{Q}}

\newcommand{\id}[1]{\mathfrak{#1}}
\newcommand{\Ko}{K}
\newcommand{\Hr}[1][n]{\mathcal{H}_{p^{#1}}}
\newcommand{\Hrg}[1][n]{H_{p^{#1}}}
\newcommand{\Hn}[1][n]{\mathcal{H}}
\newcommand{\Hng}[1][n]{H}

\newcommand{\A}[1][n]{\mathcal{A}_{p^{#1}}(\Ko)}

\newcommand{\Ub}{\overline{U}}

\newcommand{\wl}{w}
\newcommand{\tor}{\mathcal{T}_p}

\newcommand{\KoHr}[1][n]{\tilde{\Ko} \cap \Hr[#1] / \tilde{\Ko} \cap \Hr[]}

\newcommand{\UKn}[1][n]{\overline{U}_{\Ko}^{(p^{#1})}}

\newcommand{\Qn}[1][n]{\mathcal{Q}_{#1}}
\newcommand{\pari}{PARI/GP}
\DeclareMathOperator*{\gal}{Gal}
\DeclareMathOperator*{\aut}{Aut}
\DeclareMathOperator*{\ex}{exp}

\newcommand{\BibTeX}{{\scshape Bib}\kern-.08em\TeX}
\newcommand{\T}{\S\kern .15em\relax }
\newcommand{\AMS}{$\mathcal{A}$\kern-.1667em\lower.5ex\hbox
        {$\mathcal{M}$}\kern-.125em$\mathcal{S}$}

\title{Computing the torsion of the $p$-ramified module of a number field}
\date {}
\author{Frédéric PITOUN and Firmin VARESCON}

\begin{document}
\renewcommand{\proofname}{Proof.}
\maketitle 

\textbf{Abstract.} We fix a prime number $p$ and a number field $\Ko$, and denote by $M$ the maximal abelian  $p$-extension of $\Ko$ unramified outside $p$. Our aim is to study the $\Z_p$-module $\X=\gal(M/\Ko)$ and to give a method to effectively compute its structure as a $\Z_p$-module. We also give numerical results, for real quadratic fields, cubic fields and quintic fields, together with their interpretations via Cohen-Lenstra heuristics.
 
\section{Introduction}
 
 We fix a prime number $p$ and a number field $\Ko$. We denote by $M$ the maximal abelian  $p$-extension of $\Ko$ unramified outside $p$. The aim of this paper is to study the $\Z_p$-module $\X=\gal(M/\Ko)$ and give an algorithm to compute its $\Z_p$-structure. This module is described by the exact sequence 
 \begin{equation}
    \label{se:Xo}
    \xymatrix{
 \Ub_{\Ko} \ar[r] & \prod_{v |p} U_v^1 \ar[r] & \X \ar[r] & \gal(\Hn/\Ko) \ar[r] & 1,
 }
  \end{equation}
from class field theory (\cite[p. 294]{gra2}), where $\overline{U_K}$ is the pro-$p$-completion of the group of units $U_K$, $U_v^1$ is the group of principal units at the place $v$ above $p$ of $\Ko$, and $\Hn$ is the maximal $p$-sub-extension of the Hilbert class field of $\Ko$. Leopoldt's conjecture for $\Ko$ and $p$ is equivalent to injectivity of $\Ub_{\Ko} \rightarrow \prod_{v |p} U_v^1$. Therefore, from this exact sequence, we deduce that the $\Z_p$-rank $r$ of $\X$ is greater or equal to $r_2+1$ and is equal $r_2+1$ if and only if $K$ and $p$ satisfy Leopoldt's conjecture. Hence $\X$ is the direct product of a free part isomorphic to $\Z_p^{r}$ and of a torsion part, that we denote by $\tor$. Our algorithm checks whether $\Ko$ satisfies Leopoldt's conjecture at $p$ and then computes the torsion $\tor$.\\
 
 We propose a method which is based on the fact that the $\Z_p$-module $\X$ is the projective limit of the $p$-parts of the ray class groups modulo $p^n$, $\A$. We then study the stabilization of these groups with respect to $n$ and the behaviour of invariants of $\A$, as $n$ is increasing. This approach leads us to our algorithm.\\
 
 Before addressing the technical part of this article, we recall the definition and some basic properties of the ray class groups modulo  $p^n$. Then, we use our algorithm to compute some cases and propose an heuristic explanation of the statistical data, using the Cohen-Lenstra philosophy (\cite{HCL}).\\

\section{Background from class field theory} 
In this section, we recall the basic notions from class field theory that we will need later. We fix $v$ a place of $\Ko$ above $p$ and $\pi_v$ a local uniformiser of $K_v$, the completion of $\Ko$ at $v$. We use \cite{gra2} and \cite{serre} as main references.

\begin{definition}\begin{enumerate} ~
\item The conductor of an abelian extension of local fields $L_v/K_v$ is the minimum of integers $c$ such that $U_v^c \subset N_{L_v/K_v}(L_v^\times)$ (we recall that $U_v^c=1+(\pi_v^c)$ and we use the convention $U_v^0=U_v$).
\item(Theorem and Definition 4.1 + Lemma 4.2.1 \cite{gra2} p. 126-127)
The conductor of an abelian extension $L/K$ of a global field is the ideal $\id{m}=\prod_{v} \id{p}_v^{c_v}$, where $v$ runs through all finite places of $K$ and where $c_v$ is the conductor of the local extension $L_v/K_v$.
\end{enumerate}
\end{definition}

We start with two lemmas.
\begin{lemma}[Proposition 9 p. 219 \cite{serre}] \label{lemmeserre}
Let $K_v$ be the completion of $K$ at the valuation $v$ normalized by $v(p)=1$ and $v(\pi_v)=\frac{1}{e_v}$, where $e_v$ is the ramification index of the extension $K_v/\Q_p$. If $m > \frac{e_v}{p-1}$, then the map $x \mapsto x^p$ is an isomorphism from $U_v^{m}$ to $U_v^{m+e_v}$.
\end{lemma}

\begin{lemma} \label{lemme_cond_loc}
    Let $K_v \subset L_v \subset M_v$ be a tower of  extensions of $\Q_p$,  such that the extension $M_v/K_v$ is abelian and the extension $M_v/L_v$ is of degree $p$. We denote respectively by $c_{M,v}$ and $c_{L,v}$ the conductors of the extensions $M_v/K_v$ and $L_v/K_v$. If  $c_{L,v} > \frac{e_v}{p-1}$, then we have
\begin{equation} \label{inq}
c_{M,v} \leq c_{L,v} +e_v.
\end{equation}

  \end{lemma}
  \begin{proof}

By definition $c_{L,v}$ is the smallest integer $n$ such that $U_v^{n} \subset N_{L_v/K_v}(L_v^\times)$. Local class field theory gives the diagram
$$
\xymatrix{
1 \ar[r] & N_{M_v/K_v}(M_v^\times) \ar[r] \ar@{^(->}[d]& K_v^\times \ar[r] \ar@{=}[d] & \gal(M_v/K_v) \ar[r] \ar@{->>}[d] & 1 \\
1 \ar[r] & N_{L_v/K_v}(L_v^\times) \ar[r] & K_v^\times \ar[r] & \gal(L_v/K_v) \ar[r] & 1 \\
}
$$
Applying the snake lemma we get the exact sequence
$$
\xymatrix{
1 \ar[r] & N_{M_v/K_v}(M_v^\times) \ar[r] &  N_{L_v/K_v}(L_v^\times) \ar[r] & \gal(M_v/L_v)=\Z/p\Z \ar[r] & 1.
}
$$
Consequently $N_{M_v/K_v}(M_v^\times)$ is a subgroup of $N_{L_v/K_v}(L_v^\times)$ of index $p$.
Let $n \in \mathbb{N},n \geq c_{L,v}+e_v$ and  $x \in U_v^{n}$. We have to show that $x \in N_{M_v/K_v}(M_v^\times)$.
By Lemma \ref{lemmeserre}, $x^\frac{1}{p}$ is a well-defined element of $U_v^{n-e_v}$. Yet $n-e_v \geq c_{L,v}$ therefore $x^\frac{1}{p} \in N_{L_v/K_v}(L_v^\times)$. Now, as $ N_{M_v/K_v}(M_v^\times)$ is of index $p$ in $N_{L_v/K_v}(L_v^\times)$, we deduce that $x \in  N_{M_v/K_v}(M_v^\times)$. We have therefore that $U_v^{n} \subset N_{M_v/K_v}(M_v^\times)$ for all integers $n$ such that $n \geq c_{L,v}+e_v$. By the definition of the conductor, this proves that \ref{inq}.

\end{proof}

\begin{definition} Let $n$ be a positive integer.
  We denote by
  \begin{itemize}
  \item[$\bullet$] $\Hng$ the maximal abelian unramified extension of $\Ko$;
  \item[$\bullet$] $\Hrg$ the compositum of all abelian extensions of $\Ko$ whose conductors divide $p^n$;
  \item[$\bullet$] $\Hr$ the compositum of all abelian $p$-extensions of $\Ko$ whose conductors divide $p^n$;
  \item[$\bullet$] $M$ the maximal extension of $\Ko$ which is abelian and unramified outside~$p$.
  \end{itemize}
So the Galois groups $\gal(\Hn/\Ko)$ and $\gal(\Hr/\Ko)$ are respectively isomorphic to the $p$-parts of $\gal(\Hng/\Ko)$ and $\gal(\Hrg/\Ko)$.\\
\end{definition}

\begin{proposition}[Corollary 5.1.1 p. 47 \cite{gra2}]
  We have the exact sequences
$$
\xymatrix{
1 \ar[r] & K^\times \prod_{v \nmid p} U_v \prod_{v |p} U_v^{ne_v} \ar[r] & \mathcal{I}_{\Ko} \ar[r] & \gal(\Hrg/\Ko) \ar[r] & 1
}
$$
$$
\xymatrix{
1 \ar[r] & K^\times \prod_{v } U_v  \ar[r] & \mathcal{I}_{\Ko} \ar[r] & \gal(\Hng/\Ko) \ar[r] & 1,
}
$$
where $\mathcal{I}_{\Ko}$ is the group of id\`eles of $\Ko$.

\end{proposition}
  We denote the Galois group $\gal(\Hr/\Ko)$ by $\A$. It is the $p$-part of the Galois group $\gal(\Hrg/\Ko)$ which, in turn, is isomorphic to the ray class group modulo $p^n$ of $\Ko$.
By definition, we have a natural inclusion $\Hr \subset \Hr[n+1]$, the union $\displaystyle \bigcup_{n} \Hr$ is equal to $M$ and the projective limit $\varprojlim \A$ is canonically isomorphic to $\X$.

\begin{proposition} \label{seX0}
  For any integer $n>0$, the Galois groups of the extensions $M$ and $\Hrg$ of $\Ko$ are related by the exact sequence
$$
\xymatrix{
1 \ar[r] & U_{\Ko}^{(p^n)} \ar[r] & \prod_{v|p} U_v^{ne_v} \ar[r] & \gal(M/\Ko) \ar[r] & \gal(\Hrg/\Ko) \ar[r]  & 1,
}
$$
where $U_{\Ko}^{(p^n)}=\{u \in U_{\Ko} \text{ such that }  \forall v |p,\ \ u \in U_v^{ne_v}\}$ and

$$\xymatrix{
 \UKn \ar[r] & \prod_{v|p} U_v^{ne_v} \ar[r] & \X \ar[r] & \A \ar[r] & 1,
}$$
where $\Ub_{\Ko}^{(p^n)}$ is the pro-$p$-completion of $U_{\Ko}^{(p^n)}$, i.e $\displaystyle \varprojlim_m U_{\Ko}^{(p^n)}/p^m$. If moreover $\Ko$ and $p$ satisfy Leopoldt's conjecture, then $\UKn \rightarrow \prod_{v|p} U_v^{ne_v}$ is injective.
\end{proposition}

\begin{proof} To obtain the second exact sequence, we apply the pro-$p$-completion process to the first. Note that the injectivity of $\UKn \rightarrow \prod_{v|p} U_v^{ne_v}$ is equivalent to Leopoldt's conjecture. Now we prove exactness of the first sequence.\\
From the definition of the extensions $M$ and $\Hrg$, we deduce the commutative diagram
$$\xymatrix{
1 \ar[r] &\Ko^\times \prod_{v \nmid p} U_v \prod_{v|p} 1  \ar[r] \ar@{^(->}[d] & \mathcal{I}_{\Ko} \ar[r] \ar@{=}[d] &\gal(M/\Ko) \ar[r]  \ar@{->>}[d] & 1\\
1 \ar[r] & \Ko^\times \prod_{v \nmid p} U_v \prod_{v |p} U_v^{ne_v}  \ar[r] & \mathcal{I}_{\Ko} \ar[r] & \gal(\Hrg/\Ko)  \ar[r] & 1.
}$$

By the snake lemma, we have that

$$\ker(\gal(M/\Ko) \to  \gal(\Hrg/\Ko)) = (K^\times \prod_{v \nmid p} U_v \prod_{v |p} U_v^{ne_v})/ (K^\times \prod_{v \nmid p} U_v \prod_{v|p} 1).$$

Now, we define the map
$$
\theta: (K^\times \prod_{v \nmid p} U_v \prod_{v |p} U_v^{ne_v}) \to (\prod_{v|p} U_v^{ne_v}) / U_{\Ko}^{(p^n)},
$$

by setting for $k(u_v)_v \in K^\times \prod_{v \nmid p} U_v \prod_{v |p} U_v^{ne_v}$,
$ \theta(k(u_v)_v)=\ov{(u_v)_{v|p} }$, where $\ov{(u_v)_{v|p}}$ is the class of $(u_v)_{v|p}$ in $(\prod_{v|p} U_v^{ne_v})/  U_{\Ko}^{(p^n)}$.\\

We first check that the map $\theta$ is well defined, i.e. that if $k(u_v)_v=k'(u_v')_v$ in $K^\times\prod_{v \nmid p} U_v \prod_{v |p} U_v^{ne_v}$, then $\theta(k(u_v)_v)=\theta(k'(u_v')_v)$. By definition, for all $v$, $k(u_v)_v=k'(u_v')_v$ if and only if $i_v(k)u_v=i_v(k')u_v'$, where $i_v$ is the embedding of $K$ in $K_v$. We deduce that  for all $v$, $i_v(k'k^{-1}) \in U_v$ and that for all $v|p$, $i_v(k'k^{-1}) \in U_v^{ne_v}$. So we get $k'k^{-1} \in U_{\Ko}^{(p^n)}$ and $\ov{(u_v)_{v|p}}=\ov{(u_v')_{v|p}}$.\\

 It is clear that $(K^\times \prod_{v \nmid p} U_v \prod_{v|p} 1) \subset \ker(\theta)$ and that the map $\theta$ is surjective. We will show that $(K^\times \prod_{v \nmid p} U_v \prod_{v|p} 1) = \ker(\theta)$. Let $k(u_v) \in \ker(\theta)$. Then there exists an $x \in U_{\Ko}^{(p^n)}$ such that for all $v|p$, $u_v=i_v(x)$. We consider the element $x(u_v')_v$, where $u_v'=1$ if $v|p$ and $u_v'=i_v(x)^{-1}u_v$ if $v \nmid p$. We have $(u_v)_v=x(u_v')_v \Rightarrow k(u_v)_v=kx(u_v')_v$ and as $kx(u_v')_v \in (K^\times \prod_{v \nmid p} U_v \prod_{v|p} 1)$, we have  $\ker(\theta) \subset (K^\times \prod_{v \nmid p} U_v \prod_{v|p} 1) $ and finally
$$
(K^\times \prod_{v \nmid p} U_v \prod_{v |p} U_v^{ne_v})/ (K^\times \prod_{v \nmid p} U_v \prod_{v|p} 1) \simeq  (\prod_{v|p} U_v^{ne_v})/ U_{\Ko}^{(p^n)}.
$$
We deduce the first exact sequence.
\end{proof}

\section{Explicit Computation of  $\tor$}
In this section, we present our method to check that $\Ko$ satisfies Leopoldt's conjecture at $p$ and then to compute $\tor$. The main point is that, for $n$ large enough, $\A$ determines $\X$.

\subsection{Stabilization of $\A$}
For simplicity we denote $Y_n=\ker(\A[n+1] \to \A[n])$. Let $\tilde{\Ko}$ be the compositum of all the $\Z_p$-extensions of $\Ko$. 
We denote by $r$ the $\Z_p$-rank of $\X$, so that $r \geq r_2+1$.

\begin{proposition} \label{surjKCr}
There exists an $n_0$ such that $\KoHr[n_0]$ is ramified at all places above $p$.  Also, for all $n \geqslant n_0$,  $ Y_n$ surjects onto $ (\Z/p\Z)^{r}$.
\end{proposition}

Before proving the proposition, we need a lemma.
\begin{lemma} \label{lsurj}
If the extension $\KoHr[n]$ is ramified at a place $v$ above $p$, then $c_{n,v} > \frac{e_v}{p-1}$, where $c_{n,v}$ is the conductor of the local extension $(\tilde{\Ko} \cap \Hr[n])_w/\Ko_v $ and $w$ is a place above $v$.
\end{lemma}

\begin{proof}[Proof of Lemma \ref{lsurj}.]
As $M$ contains the cyclotomic $\Z_p$-extension, there exists an $n_0$ such that $\KoHr[n_0]$ is ramified at all places $v$ above $p$. As $\KoHr[n_0]$ is ramified at $v$ then, for $n \geq n_0$, $\KoHr[n]$ is ramified at $v$, so that there exists an $m$ such that $n \geq m \geq 2 $ and that $\KoHr[m-1]$ is unramified at $v$ and such that $\KoHr[m]$ is ramified at $v$. Then the local conductor $c_{m,v}$ is greater than $(m-1)e_v$, yet $m \geq 2$ so $c_{m,v} > (m-1)e_v \geqslant e_v \geqslant \frac{e_v}{p-1}$. As the conductor of the extension $\tilde{\Ko} \cap \Hr[m] / \Ko$ divides the conductor of $\tilde{\Ko} \cap \Hr[n] / \Ko$, we have $c_{n,v} \geqslant c_{m,v}>\frac{e_v}{p-1}$. 
\end{proof}
\begin{proof}[Proof of the Proposition \ref{surjKCr}.]

We consider the diagram
\begin{equation}
  \label{diag_ker}
  \xymatrix{
\tilde{\Ko} \cap \Hr[n] \ar@{-}[r] &  (\tilde{\Ko} \cap \Hr[n]) \Hr[] \ar@{-}[rr] & &\Hr[n] \\
\tilde{\Ko} \cap \Hr[n-1] \ar@{-}[r] \ar@{-}[u] &  (\tilde{\Ko} \cap \Hr[n-1]) \Hr[] \ar@{-}[r] \ar@{-}[u] &  \Hr[n-1] \ar@{-}[ur] \ar@/_1pc/@{-}[ur]_{Y_{n-1}} \\
\tilde{\Ko} \cap \Hr[] \ar@{-}[r] \ar@{-}[u] & \Hr[]\ar@{-}[u] \ar@{-}[ru] &   \\
\Ko. \ar@{-}[u] &
}
\end{equation}

We have $\gal(\tilde{\Ko}/\Ko) = \Z_p^{r}$.
It is clear that $Y_n \surj \gal(\tilde{\Ko} \cap \Hr[n+1] / \tilde{\Ko} \cap \Hr[n])$. Yet $\gal(\tilde{\Ko}/\tilde{\Ko} \cap \Hr[n])$ is a $\Z_p$-submodule of $\gal(\tilde{\Ko}/\Ko) = \Z_p^{r}$ of finite index, so it is isomorphic to $\Z_p^{r}$. Hence there exist $r$ extensions, say $M_1,M_2,\cdots,M_{r}$ of $\tilde{\Ko} \cap \Hr$, contained in $\tilde{\Ko}$ such that $\gal(M_i/\tilde{\Ko} \cap \Hr) \simeq \Z/p\Z$ and $\gal(M_1\cdots M_{r}/\tilde{\Ko} \cap \Hr) \simeq (\Z/p\Z)^{r}$.
Yet the conductor of the extension $\tilde{\Ko} \cap \Hr /\Ko$ divides $p^n=\prod_{v |p} \id{p}_v^{ne_v}$. Moreover the hypothesis on $\KoHr[n]$ ensures that we can use Lemma \ref{lemme_cond_loc} and consequently the conductor of the extension $M_i/\Ko$ divides $\prod_{v | p} \id{p}_v^{ne_v+e_v} =p^{n+1}$, i.e., $M_i \subset \Hr[n+1]$ for all $i \in \{1,\cdots,r\}$. Hence the map is surjective.
\end{proof}

We deduce immediately the corollary.
\begin{corollary}
Let $n$ be a positive integer such that the extension
$\tilde{\Ko} \cap \Hr[n] / \tilde{\Ko} \cap \Hr[]$ is ramified at all places above $p$, and that the cardinality of $Y_n$ is exactly $p^{r_2+1}$. Then 
$Y_n \simeq (\Z/p\Z)^{r_2+1}$ and $\Ko$ satisfies the Leopoldt's conjecture at $p$. 
\end{corollary}

From now on, as we can numerically check that $\Ko$ satisfies the Leopoldt's conjecture at $p$, we assume it does so, in
order to compute $\tor$. Note that if Leopoldt's conjecture is false, then $r>r_2+1$ and our algorithm never stops.

\begin{corollary}
 We assume that, for some integer $n$ such that the extension
$\tilde{\Ko} \cap \Hr[n] / \tilde{\Ko} \cap \Hr[]$ is ramified at all places above of $p$, the cardinal of $Y_n$ is exactly $p^{r_2+1}$. Then $Y_n \simeq \gal(\tilde{\Ko} \cap \Hr[n+1]/\tilde{\Ko} \cap \Hr)$.
\end{corollary}

It remains to check that if $Y_{n_0} \simeq (\Z/p\Z)^{r_2+1}$ for some $n_0$, then  $Y_n \simeq(\Z/p\Z)^{r_2+1}$ for all integers $n \geq n_0$.
For this purpose, we consider the exact sequence defining the $p$-part of the ray class group:
$$
\xymatrix{
1 \ar[r] & \UKn \ar[r] & \prod_{v|p} U_v^{ne_v} \ar[r] & \X \ar[r] & \A \ar[r] & 1,
}
$$
and we denote $\mathcal{Q}_n= \prod_{v|p} U_v^{ne_v}/ \UKn$. We have $\mathcal{Q}_n=\gal(M/\Hr)$ and consequently $\mathcal{Q}_{n}  / \mathcal{Q}_{n+1} = Y_n \simeq \gal(\Hr[n+1]/\Hr) $. 

 \begin{proposition}
  For $n \geq 2 $, raising to the $p^{\text{th}}$ power induces, via the Artin map, a surjection from $Y_n$ to $Y_{n+1}$.
 \end{proposition}

 \begin{proof}
Recall that $\mathcal{Q}_n=\prod_{v |p} U_v^{ne_v} / \UKn=\ker(\X \to \A)$. We have that $n > \frac{1}{p-1}$. Raising to the $p^{th}$ power realizes an isomorphism of $\prod_{v |p} U_v^{ne_v}$ onto $\prod_{v |p} U_v^{ne_v+e_v}$. This isomorphism induces a surjection from $\mathcal{Q}_n$ onto $\mathcal{Q}_{n+1}$. We consider the diagram
$$
\xymatrix{
1 \ar[r] & \mathcal{Q}_{n+1} \ar@{->>}[d]^{(.)^p}\ar[r] & \mathcal{Q}_n \ar@{->>}[d]^{(.)^p} \ar[r] & \ar@{->}[d]^{(.)^p} \mathcal{Q}_n/\mathcal{Q}_{n+1} \ar[r] & 1 \\
1 \ar[r] & \mathcal{Q}_{n+2} \ar[r] & \mathcal{Q}_{n+1} \ar[r] & \mathcal{Q}_{n+1}/\mathcal{Q}_{n+2} \ar[r] & 1. \\
}
$$
We deduce from the snake lemma that the vertical arrow on the right-hand side is a surjection from $\mathcal{Q}_n/\mathcal{Q}_{n+1}$ onto $\mathcal{Q}_{n+1}/\mathcal{Q}_{n+2}$, i.e., from $Y_{n}$ onto $Y_{n+1}$.
 \end{proof}

 \begin{corollary}
We denote $q_n=\#(Y_n)$. For all $n \geqslant 2$, $q_n\geqslant q_{n+1}$. Therefore the sequence $(q_n)_{n\geq1}$ is ultimately constant.
 \end{corollary}

We recall that $Y_n$ is $\ker(\A[n+1] \to \A[n])$.
\begin{theorem} \label{th_stab_ker} As we assume Leopoldt's conjecture, there exists an integer $n_0$ such that $Y_{n_0} \simeq (\Z/p\Z)^{r_2+1}$. Moreover for all integers $n \geq n_0$, the modules $\Qn=\gal(M/\Hr)$ are $\Z_p$-free of rank $r_2+1$ and 
$$Y_n \simeq (\Z/p\Z)^{r_2+1}.$$
\end{theorem}

\begin{proof}
The $\Z_p$-module $\X$ is isomorphic to the direct product of its torsion part and of $\Z_p^{r_2+1}$. An isomorphism being chosen, we can identify $\Z_p^{r_2+1}$ with a subgroup of $\X$ and therefore define, via Galois theory, an extension $M'$ of $\Ko$ such that $\gal(M'/\Ko) \simeq \tor$ and $\tilde{\Ko}M'=M$.
 
This extension being unramified outside $p$, there exists an integer $n_1$ such that $M' \subset \Hr[n_1]$ and consequently $\Hr[n_1] \tilde{\Ko} =M$. Moreover, for all integer $n \geq n_1$, $\gal(M/\Hr)$ is a submodule of finite index of $\gal(M/M')=\Z_p^{r_2+1}$, and consequently $\Qn=\gal(M/\Hr) \simeq \Z_p^{r_2+1}$. The $\Z_p$-module $\Qn$ is therefore free of rank $r_2+1$.

About the other kernel $Y_n$ we saw that there exists an integer $n_2$ such that $Y_n$ maps surjectively onto $(\Z/p\Z)^{r_2+1}$ for all integer $n \geq n_2$ (we can choose $n_2$ to be the minimum of all integers $n$ such that for all $p$-places $v$ the conductors of $(\tilde{\Ko} \cap \Hr)_w/\Ko_v$ are at least $ \frac{e}{p-1}$). Then we note that mapping $x \in U_v^{ne_v}$ to $x \in U_v^{ne_v+e_v}$ realizes an isomorphism between $U_v^{ne_v}$ and $U_v^{ne_v+e_v}$, so that the quotient $\mathcal{Q}_{n}/\mathcal{Q}_{n+1}$, which is isomorphic to $Y_n$, is killed by $p$. Define $n_0=\text{Max}(n_1,n_2)$ and let $n \geqslant n_0$ be an integer. The kernel $Y_n$ is therefore a quotient of $\Z_p^{r_2+1}$, which maps surjectively onto $(\Z/p\Z)^{r_2+1}$ and is killed by $p$. Hence we get $Y_n \simeq (\Z/p\Z)^{r_2+1}$.
\end{proof}

\subsection{Computing the invariants of $\tor$}
We start by recalling the definition of the invariant factors of an abelian group $G$.
\begin{definition}
Let $G$ be a finite abelian group, there exists a unique sequence $a_1,\cdots,a_t$ such that $a_{i} | a_{i+1}$ for $i \in \{1,\cdots,t-1\}$ and $G \simeq \prod_{i=1}^t \Z/a_i\Z$.\\ 
These $a_i$ are the invariant factors of the group $G$. 
\end{definition}
In what follows we will denote these invariants by $\mathcal{FI}(G)=[a_1,\cdots,a_t]$. If $G$ is a $p$-group, these invariant factors are all powers of $p$.
In practice, we are able to determine the invariant factors of $\A$. We will see in this section that the knowledge of invariant factors of $\A$, for $n$ large enough, combined with the stabilizing properties of $\A$, does determine explicitly the invariants
factors of $\tor$, and thus $\tor$ itself. We recall that for $n$ large enough, $\A$ is isomorphic to the direct product of $\gal(\tilde{\Ko} \cap \Hr/\Ko)$ and of $\gal(\Hr/\tilde{\Ko} \cap \Hr)=\tor$. So we will first explore the structure of $\gal(\tilde{\Ko} \cap \Hr/\Ko)$.

\begin{proposition}
  Let $n_0$ be such that $\KoHr[n_0]$ is ramified at all places above of $p$ and 
$$Y_{n_0} \simeq (\Z/p\Z)^{r_2+1}.$$
Then for all integer $n \geq n_0$, we have
 $$\gal(\tilde{\Ko}/ \tilde{\Ko} \cap \Hr[n+1])=p\gal(\tilde{\Ko}/\tilde{\Ko} \cap \Hr[n]).$$

\end{proposition}

\begin{proof}
By Theorem \ref{th_stab_ker}, on the one hand, $\Qn$ is $\Z_p$-free of rank $r_2+1$ and on the other hand $Y_n =\Qn/\Qn[n+1]\simeq (\Z/p\Z)^{r_2+1}$. This gives $\Qn[n+1]=p\Qn$. As $\KoHr[n_0]$ is ramified at all places above $p$ and 
$Y_{n_0} \simeq (\Z/p\Z)^{r_2+1}$, we have $\tor \subset \A[n_0]$, so $\tilde{\Ko} \Hr[n_0]=M $. Then, considering the diagram
\def\trait{\ar@{-}}
$$
\xymatrix{
\tilde{\Ko} \trait[d] \trait[r] & M \trait[d] \\
\tilde{\Ko} \cap \Hr[n+1] \trait[r] \trait[d] & \Hr[n+1] \trait[d]  \ar@/^1pc/@{-}[u]^{\Qn[n+1]} \\
\tilde{\Ko} \cap \Hr \trait[r] \trait[d] & \Hr  \ar@/_2pc/@{-}[uu]_{\Qn} \\
\Ko
}
$$
we get the required isomorphism.
\end{proof}

\begin{corollary}
  Let $n_0$ be an integer such that $\KoHr[n_0]$ is ramified at all places above $p$ and such that $Y_{n_0} \simeq (\Z/p\Z)^{r_2+1}$. Then for all integers $n \geq n_0$ the invariant factors of $\gal( \tilde{\Ko} \cap \Hr[n+1]/ \Ko)$ are obtained by multiplying by $p$ each invariant factor of $\gal( \tilde{\Ko} \cap \Hr/ \Ko)$.
\end{corollary}

From the fact that $\X \simeq \Z_p^{r_2+1} \times \tor$, the ray class group, $\gal(\Hr/\Ko)$, is isomorphic to the direct product of $\gal(\tilde{\Ko} \cap \Hr/\Ko)$ and $\gal(\Hr/\tilde{\Ko} \cap \Hr)$. The invariant factors of $\gal(\Hr/\Ko)$ are then simply obtained by concatenating the two groups forming the direct product.
We now state the result that explicitly determines $\tor$.

\begin{theorem}
  Let $n$ such that $Y_n=(\Z/p\Z)^{r_2+1}$ and $\KoHr[n]$ is ramified at all places above $p$. We assume that 
$$
\mathcal{FI}(\A[n])=[b_1,\cdots,b_{t},a_1,\cdots,a_{r_2+1}]
$$
 with $(v_p(a_{1}))>(v_p(b_t))+1$, and that
$$
\mathcal{FI}(\A[n+1])=[b_1,\cdots,b_{t},pa_1,\cdots,pa_{r_2+1}].
$$
 Then we have
 $$\mathcal{FI}(\tor)=[b_1,\cdots,b_t].$$
\end{theorem}

\begin{proof}
  Indeed, as $Y_n \simeq (\Z/p\Z)^{r_2+1},$ we have $\A[i]\simeq \tor \times \gal(\tilde{\Ko} \cap \Hr[i]/\Ko)$ for $i \in \{n,n+1\}$. We saw that the invariant factors of $\gal(\tilde{\Ko} \cap \Hr[n+1]/\Ko)$ are exactly equal to $p$ times those of $\gal(\tilde{\Ko} \cap \Hr[n]/\Ko)$. Consequently, if $a$ is an invariant factor of $\gal(\tilde{\Ko} \cap \Hr[n+1]/\Ko)$, we have necessarily that $a=p a_i$ or $a=pb_i$. But as $\mathrm{Min}(v_p(a_i))>\mathrm{Max}(v_p(b_i))+1$, none of the invariants factors of $\gal(\tilde{\Ko} \cap \Hr[n+1]/\Ko)$ is of the form $pb_i$. The invariant factors of $\gal(\tilde{\Ko} \cap \Hr[n+1]/\Ko)$ are therefore exactly $pa_1,\cdots,pa_{r_2+1}$. The result follows from the fact that $\A[n+1]$ is isomorphic to the direct product of $\tor$ and $\gal(\tilde{\Ko} \cap \Hr[n+1]/\Ko)$.
\end{proof}

\section{Explicit computation of bounds}
More generally, if we denote by $e=\max_{v|p}\left\lbrace e_v\right\rbrace $ the ramification index of $\Ko/\mathbb{Q}$ and by $s$ the $p$-adic valuation of $e$, then we can start to check whether $\A$ stabilizes from rank $n=2+s$. To show that $n=2+s$ is the proper starting point we consider the diagram

$$\xymatrix{
&\tilde{\Ko} \cap \Hr[s+2]\ar@{-}[rrr]&  &  &\ar@{-}[ld]{\Hr[s+2]} \\
K_{s+1} \ar@{-}[ur] \ar@{-}[rdd]& & &\ar@{-}[ld]{\Hr[s+1]}\\
&\tilde{\Ko} \cap \Hr[] \ar@{-}[r] \ar@{-}[uu] \ar@{-}[d] & \Hr[]& \\
&\Ko
}$$
where $K_j$ is the $j^{\text{th}}$field of the $\mathbb{Z}_p$-extension of $\Ko$.\\  

We prove below that the places above $p$ are totally ramified in $K_{s+1}/K_{s}$. Therefore $\KoHr[s]$ is ramified at all places above $p$ and we start the computation by checking whether $\A$ stabilizes from $n=s+2$, and until it stabilizes.
We first prove that all places above $p$ are totally ramified in $K_{s+1}/K_s$.\\
Considering the diagram
$$
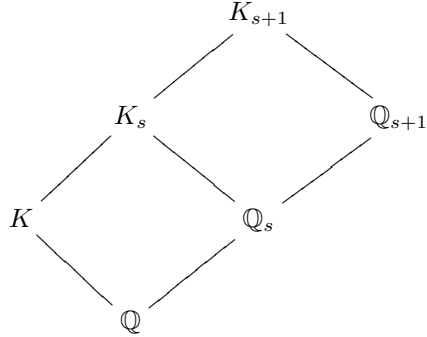
\xymatrix{
& & K_{s+1} \ar@{-}[ld] \ar@{-}[rd]&\\
&K_s& &\mathbb{Q}_{s+1} \ar@{-}[ld]\\
\Ko \ar@{-}[ru]  &  &\mathbb{Q}_s \ar@{-}[ul]  \\
& \mathbb{Q} \ar@{-}[ul] \ar@{-}[ur]
}$$
 
The ramification index of $p$ in $\mathbb{Q}_{s+1} / \mathbb{Q}$ is $p^{s+1}$, while the one in $\Ko/\mathbb{Q}$ is $p^sa$ with $p\nmid a $. Therefore the extension $K_{s+1}/\Ko$ is ramified and $K_{s+1}/K_s$ is totally ramified at all places above $p$.

\begin{corollary} Let $e$ be the ramification index of $p$ in $\Ko/\mathbb{Q}$ and $s$ be the $p$-adic valuation of $e$. Let $n \geq 2+s$, we assume that
$$
\mathcal{FI}(\A)=[b_1,\cdots,b_{t},a_1,\cdots,a_{r_2+1}],
$$
 with $\mathrm{Min}(v_p(a_i))>\mathrm{Max}(v_p(b_i))+1$, and moreover that
$$
\mathcal{FI}(\A[n+1])=[b_1,\cdots,b_{t},pa_1,\cdots,pa_{r_2+1}].
$$
 Then we have
 $$\mathcal{FI}(\tor)=[b_1,\cdots,b_t].$$

\end{corollary}

All the computations have been done using the \pari system \cite{pari}.
\begin{example}
We consider the field $K=\Q(\sqrt{-129})$ and $p=3$. We have: $\mathcal{FI}(\A[2])=[3,3,9]$, $\mathcal{FI}(\A[3])=[3,9,27]$ and $\mathcal{FI}(\A[4])=[3,27,81]$. We deduce that $\tor=(\Z/3\Z)$.
\end{example}
\section{Numerical results}
In the section, we present some of our numerical results and give an explanation of these computations. 
\subsection{Heuristic approach}
We first recall some results on Cohen-Lenstra Heuristics. The main reference on the subject is the seminal paper of Cohen-Lenstra \cite{HCL}. See also \cite{TCD}. These heuristics leads us to compare the proportion of fields with non-trivial $\tor$ with the proportion of groups with non-trivial $p$-part inside all finite abelian groups. If we assume that the extension $\Ko/\Q$ is Galois with $\Delta=\gal(K/\Q)$, then the module $\tor$ is a $\Z[\Delta]$-module. In this section, we assume that $\Delta$ is cyclic of cardinality $l$, for some prime number $l$. Then, as the $p$-part of the class group, $\tor$ itself is a finite $O_l$-module, where $O_l$ is the ring of integers of $\mathbb{\Q}(\zeta_l)$. This module $\tor$ is known in Iwasawa theory as the proper $p$-adic analogue of the class group. Hence it is a natural question to compute it, to examine the distribution of fields with non-trivial $\tor$, and to compare this distribution with the Cohen-Lenstra heuristics about the distribution of groups with non-trivial $p$-part inside all finite abelian groups.
 
In what follows, $O_F$ will be the ring of integers of a number field and $G$ will be a finite $O_F$-module.
In general, we know that all $O_F$-modules $G$ can be written in a non-canonical way as $\oplus_{i=1}^q O_F/\id{a}_i$, where the $\id{a}_i$ are ideals of $O_F$. Yet the Fitting ideal $\id{a}=\prod_{i=1}^q \id{a}_i$ depends only on the isomorphism class of $G$, considered as a $O_F$-module. This invariant, denoted by $\id{a}(G)$, can be considered as a generalization of the order of $G$. We also have $N(\id{a}(G))=\# G$.\\

We consider a function $g$, defined on the set of the isomorphism classes of $O_F$-modules (typically $g$ is a characteristic function). We follow \cite{HCL} for the next definition, using same notations.
\begin{definition}
 The average of $g$, if it exists, is the limit when $N \to \infty$ of the quotient 
$$
\frac{\displaystyle{\sum_{G,N(\id{a}(G)) \leq N} \frac{g(G)}{\#\aut_{O_F}(G)}}}{\displaystyle{\sum_{G,N(\id{a}(G)) \leq N} \frac{1}{\#\aut_{O_F}(G)}}}.
$$
where $\displaystyle{\sum_{G, N(\id{a}(G)) \leq N}}$ is the sum is over all isomorphism classes of $O_F$-modules $G$.
This average is denoted by $M_{l,0}(g)$.
\end{definition}

We denote by $\wl(\id{a})= \displaystyle{\sum_{G,\id{a}(G)=\id{a}} \frac{1}{\#\aut_{O_F}(G)}}$, where $\id{a}$ is an ideal of $O_F$ (using same notation as \cite{HCL}).\\
\begin{proposition}[Corollary 3.8 p.40 \cite{HCL}]\label{propw}
  Let $n \in \mathbb{N}$. Then
$$
\wl(\id{a})= \frac{1}{N(\id{a})}\left( \prod_{\id{p}^\alpha || \id{a}} \prod_{k=1}^\alpha (1-\frac{1}{N_{O_\Ko}(\id{p})^k}) \right)^{-1}.
$$
The notation $\id{p}^\alpha || \id{a}$ means that $\id{p}^\alpha | \id{a}$ and that $\id{p}^{\alpha +1} \nmid \id{a}$. Consequently the function $\wl$, defined on the set of ideals of $O_F$, is multiplicative.
\end{proposition}

\textbf{Notation.} We denote by $\Pi_p$ the characteristic function of the set of isomorphism classes of groups whose $p$-part is non-trivial.

\begin{proposition}[Example 5.10 p.47 \cite{HCL}]
 We denote by $\id{p}_1,\cdots,\id{p}_g$ the $p$-places of $O_F$, the average of $\Pi_p$ exists and we have
\begin{equation} \label{moy}
M_{l,0}(\Pi_p)=1-\prod_{i=1}^g \prod_{k \geq 1} \left( 1-\frac{1}{p^{kf_i}}\right),
\end{equation}
where $f_i$ is the degree of the residual extensions $O_F/\id{p}_i$ over $\mathbb{F}_p$.
\end{proposition}

\begin{corollary}
If the extension $F$ is a Galois extension, all residual degrees are equals to $f$ and in this case $$M_{l,0}(\Pi_p)=1-\prod_{k \geq 1}\left( 1-\frac{1}{p^{kf}}\right)^g.$$
\end{corollary}

\begin{remark}
The real number $M_{l,0}(\Pi_p)$ is called the 0-average. This notion can be generalized to the u-average. The expression to compute the u-average is obtained by replacing $k$ by $k+u$ in the expression \ref{moy} of the 0-average.
\end{remark}

Let $\mathcal{K}$ be a set of number fields, cyclic of degree $l$, let  $K$ run through $\mathcal{K}$ and let $G$ be the $p$-part of the class group of $F$. We assume $l\neq p$. If we denote by $A=\mathbb{Z}[\Delta]/\sum_{g \in \Delta}g$, where $\Delta = \gal(K/\mathbb{Q})$, it is easy to see that $G$ is a finite $A$-module. As $\Delta$ is cyclic of order $l$, then $G$ is an $O_l$-module. Following the Cohen-Lenstra Heuristics we give the assumptions.

\begin{assume}[Assumptions p.54 \cite{HCL}] Recall that $l=[\Ko:\mathbb{Q}]$, then we have: 
\begin{enumerate}
 \item (Complexe quadratic case) If $r_1=0$, $r_2=1$ then the proportion of $G$ which are non-trivial is the 0-average of $\Pi_p$, restricted to $O_l$-modules of order prime to $l$.
\item (Totally real case) If $r_1=n$, $r_2=0$ then the proportion of $G$ which are non-trivial is the 1-average of $\Pi_p$, restricted to $O_l$-modules of order prime to $l$.
\end{enumerate}

\end{assume}

\subsection{Somes numerical results}

\subsubsection{Case of the quadratic fields}
We observed that in the case of real quadratic fields the proportion of fields with non-trivial $\Z_p$-torsion of $\X$ was a 0-average, and a 1-average for the imaginary quatratic fields. We will explain why this phenomenon is consistent with Cohen-Lenstra Heuristics in Section \ref{explain}.\\
We consider all quadratic fields $\Q(\sqrt{d})$ with $d$ square-free and $ 0 < d \leq 10^9$. Then we compute the proportion of fields with non-trivial $\tor$. We denote this proportion by $f_{\ex}$. The relative error $|f_{\ex}-M_{2,0}(\Pi_p)|/M_{2,0}(\Pi_p)$ is denoted by $\delta$. We remark that $\delta$ tends to 0 if we increase the numbers of fields whose torsion we compute, except for the case $p$=2 and 3. We explain this discrepancy with 2 and 3 in Section \ref{explain}.
\label{table}
\begin{center}
\begin{tabular}[center]{|*{4}{c|}}
  \hline
   $p$ & $M_{2,0}(\Pi_p)$ & $f_{\ex}$ & $\delta$ \\
\hline
   2   &   0,71118       &  0,93650  &    0,31683    \\
  \hline
   3   &   0,43987       &  0,50120  &    0,13942    \\
  \hline
   5   &   0,23967       &  0,23854  &    0,00470  \\
  \hline
   7   &   0,16320       &  0,16280  &    0,00247     \\
  \hline
  11   &   0,09916       &  0,09893  &    0,00243   \\
  \hline
  13   &   0,08284       &  0,08266 &     0,00212    \\
  \hline
  17   &   0,06228       &  0,06214 &     0,00233   \\
  \hline
  19   &   0,05540       &  0,05526 &    0,00260   \\
  \hline
  23   &    0,04537      &  0,04527 &    0,00207    \\
  \hline
  29   &   0,03375       &  0,03560   &      0,00193     \\
  \hline
   31   &   0,03330       &  0,03323   &     0,00219      \\
  \hline
   37   &   0,02776       &  0,02770   &     0,00198      \\
  \hline
  41   &   0,02499       &  0,02493   &     0,00207      \\
  \hline
  43   &   0,02380       &  0,02376   &     0,00152      \\
  \hline
  47   &   0,02173       &  0,02168   &     0,00207      \\
 \hline
\end{tabular}
\end{center}
\vspace{5mm}

We consider now the quadratic field $\Q(\sqrt{d})$ with $-10^9 \leqslant d \leqslant 0$. One uses the 1-average denoted by $M_{2,1}(\Pi_p)$.

\vspace{5mm}

\begin{center} 
\begin{tabular}[center]{|*{4}{c|}}
  \hline
   $p$ & $M_{2,1}(\Pi_p)$ & $f_{\ex}$ & $\delta$ \\
\hline
   2  &   0,42235      &  0,93650  &    1.12734      \\
  \hline
   3   &   0,15981       &  0,25718  &    0,60926      \\
  \hline
   5   &   0,04958       &  0,04909  &    0,00989  \\
  \hline
   7   &   0,02374       &  0,02365  &    0,00374     \\
  \hline
  11   &   0,00908       &  0,00905  &    0,00416   \\
  \hline
  13   &   0,00641       &  0,00638 &     0,00360    \\
  \hline
  17   &   0,00368       &  0,00365 &     0,00445   \\
  \hline
  19   &   0,00292       &  0,00291 &    0,00589   \\
  \hline
  23   &    0,00198      &  0,00197 &    0,00510    \\
  \hline
  29   &   0,00123       &  0,00122   &      0,00916     \\
  \hline
   31   &   0,00108       &  0,00107   &     0,00929      \\
  \hline
   37   &   0,00075       &  0,00074   &     0,00813      \\
  \hline
  41   &   0,00061      &  0,00060   &     0,00982      \\
  \hline
  43   &   0,00055       &  0,00055   &     0,00998      \\
  \hline
  47   &   0,00046       &  0,00046   &     0,01626      \\
 \hline
\end{tabular} 
\end{center}

\vspace{5mm}

We have also computed the proportions for cubic fields, with the program of K. Belabas \cite{belabas}, and for quintic fields using the tables which are available on the website dedicated to \pari system \cite{pari}. Then we consider the distribution of torsion modules with respect to invariants factors that will not be presented here, for the sake of brevity. To compute $\#\aut_{O_\Ko}(G)$ we use \cite{hall}.\\

\subsubsection{Explanation of numerical results} \label{explain}
In this section we explain our numerical results. Looking at the two tables in \S \ref{table} we remark that the proportion $f_{\ex}$ for real quadratic fields seems to be a 0-average, and a 1-average for the imaginary quadratic. We remark also that the default $\delta$ for $p=2, 3$ increases with the number of fields computed. To explain these phenomena we recall a computation of Gras \cite{gra1} p. 94-97.
Let $k$ be a number field, we denote by $K=k(\zeta_p)$ and $\omega$ the idempotent associated with the action of $\gal(K/k)$ on $\mu_p$.

\begin{theorem}[Corollaire 1 p. 96 \cite{gra1}] 
Let $p$ be a prime, $p \neq 2$.
If $\mu_p \not\subset k$ then the torsion of $\X$ is trivial if and only if any prime ideal of $k$ dividing $p$ is totally split in $K/k$ and $(Cl_K)^\omega$ is trivial, where  $Cl_K$ is the $p$-part of the class group of $K$.
\end{theorem}
In the case of quadratic fields, if $p > 3 $ then  $\mu_p \not\subset k$ and the ramification index of $p$ in $\mathbb{Q}(\zeta_p)/\mathbb{Q}$ is $p-1$; then all prime ideals of $k$ dividing $p$ ramify in $K$. Therefore they are not totally split, and so the torsion is trivial if and only if $(Cl_K)^\omega$ is trivial. So when $k$ is a real quadratic field the computation of $\tor$ reduces to the computation of a class group of imaginary quadratic field and we use the 0-average following Cohen-Lenstra Heuristics. In the case of imaginary quadratic the remark \cite{gra1} p.96-97 explains the 1-average. In the case $p=3$, if $d\equiv 6$ mod 9 then the ideal of $k$ above $p$ is totally split in $K$, so the torsion is non-trivial. It explains why the frequency obtained is greater.
If we consider the other average $M'_2(\Pi_3)=M_{2,0}(\Pi_3) \times \frac{7}{8}+\frac{1}{8}$, then we obtain in the real case

\vspace{5mm}
\begin{center}
\begin{tabular}[center]{|*{4}{c|}}
  \hline
   $N$ & $M'_2(\Pi_3)$ & $f_{\ex}$ & $\delta$ \\
\hline
   $10^6$   &   0,50989       &  0,48094 &    0,05678      \\
  \hline
   $10^7$   &   0,50989       &  0,49054  &    0,03794      \\
  \hline
   $10^8$   &   0,50989       &  0,49697  &    0,02533  \\
  \hline
   $10^9$   &   0,50809       &  0,50120  &    0,01704     \\
  \hline
\end{tabular}
 
\vspace{5mm}

We now make the computation without the case $d\equiv 6$ mod 9.\\

\vspace{5mm}
\begin{tabular}[center]{|*{4}{c|}}
  \hline
   $N$ & $M_{2,0}(\Pi_3)$ & $f_{\ex}$ & $\delta$ \\
\hline
   $10^6$   &   0,43987       &  0,40679   &    0,07521   \\
  \hline
   $10^7$   &   0,43987       &  0,41776   &    0,05027      \\
  \hline
   $10^8$   &   0,43987       &  0,42511  &    0,03356  \\
  \hline
   $10^9$   &   0,43987       &  0,42995  &    0,02257     \\
  \hline
\end{tabular}
\end{center}
\vspace{5mm}
It remains to study the 9-rank in the case where $d\equiv 6$ mod 9, and to try and find density formulas for the 9-rank. 
Finally, the discrepancy in the case $p=2$ is explained by genus theory. Indeed, if the discriminant is divided by enough primes then the torsion is not trivial. This explains why the frequency tends to 1.\\
\\
\textbf{Acknowledgement.} It is our pleasure to thank Bill Allombert for his help with \pari \ computations, Christophe Delaunay for his relevant suggestions
and the anonymous referees who have devoted their time to study the previous versions of this article and suggested many improvements.

\bibliographystyle{smfalpha}

\noindent Fredéric PITOUN,\\
27 Avenue du 8 mai 1945,\\
11400 Castelnaudary, FRANCE.\\
frederic.pitoun@free.fr\\ 
\\
Firmin VARESCON,\\
Laboratoire de mathématiques de Besançon, CNRS UMR 6623,\\
Université de Franche Comté, 16 Route de Gray, 25020 Besançon Cédex, FRANCE.\\
firmin.varescon@univ-fcomte.fr\\

\end{document}